\documentclass[11pt]{amsart}
\usepackage{amsmath, amssymb,amscd}
\usepackage[mathscr]{eucal}
\usepackage{amscd}
\usepackage[all, cmtip]{xy}

\usepackage[top=2.3cm,left=2.2cm,right=2.2cm,bottom=2.3cm]{geometry}
\renewcommand{\thefootnote}{\fnsymbol{footnote}}
\newcommand{\authorfootnotes}{\renewcommand\thefootnote{\@fnsymbol\c@footnote}}%
 
\theoremstyle{plain} 
\newtheorem{theorem}{Theorem}
\newtheorem{lemma}[theorem]{Lemma}

\newtheorem{corollary}[theorem]{Corollary}

\newtheorem{conjecture}[theorem]{Conjecture}

\theoremstyle{definition} 

\theoremstyle{remark} 

\numberwithin{theorem}{section}
\numberwithin{equation}{section}

\newcommand{\texteqn}[1]{\hspace{4pt}\text{#1}\hspace{4pt}}

\newcommand{\legendre}[2]{\left(\frac{#1}{#2}\right) }
\newcommand{\comment}[1]{}

\def\e{\varepsilon}

\newcommand{\ga}{\alpha}
\newcommand{\gb}{\beta}
\newcommand{\gd}{\delta}
\newcommand{\gD}{\Delta}
\newcommand{\gs}{\sigma}

\newcommand{\mc}[1]{\mathcal{#1}}

\newcommand{\ul}[1]{\underline{#1}}
\newcommand{\upone}[1]{{#1}^{(1)}}

\def\Z{\mathbb Z}
\def\N{\mathbb N}
\def\Q{\mathbb Q}

\def\C{\mathbb C}
\def\F{\mathbb F}

\def\Re{\operatorname{Re}}
\def\deg{\operatorname{deg}}

\def\mod{\operatorname{mod}}

\def\gcd{\operatorname{gcd}}

\def\log{\operatorname{log}}

\def\sgn{\operatorname{sgn}}

\def\Box{{\operatorname{Box}}}

\begin{document}
\title[]{Distribution of the trace of Frobenius on average for rank 2 Drinfeld modules}

\author{Abel Castillo}
\address[]
{
Department of Mathematics, Statistics and Computer Science, University of Illinois at Chicago, 851 S Morgan St, 322
SEO, Chicago, 60607, IL, USA.
} 
\email[]{abel65535@gmail.com}

\begin{abstract} 
Let $q$ be an odd prime power, $a \in \F_q[T]$ and $u \in \F_q^*$. Provided $q \geq 17$, we compute the average number of primes $p$ for which the characteristic polynomial of the Frobenius at $p$ is $X^2 - aX + up$ over a family of rank 2 Drinfeld $\F_q[T]$-modules. Our results give asymptotic formulas in the $x$-limit. \end{abstract}
\maketitle

\section{Introduction}

Let $q$ be a power of an odd rational prime. Let $\F$ be the field with $q$ elements. Let $A:=\F[T]$ and $F:=\F(T)$. Since $A$ is a unique factorization domain, non-zero prime ideals correspond to monic irreducible polynomials in $T$. We always use the letter $p$ to denote monic irreducibles in $A$. For $a \in A$, define $|a| := q^{\deg a}$ and $\sgn a$ to be the leading coefficient of $n$ as a polynomial in $T$.

Let $F\{\tau\}$ be the non-commutative polynomial ring with the commutation rule $\tau a = a^q \tau$ for all $a \in F$. A Drinfeld module (more precisely, a Drinfeld $A$-module) over $F$ of rank $r$ is given by an $\F$-algebra homomorphism $\Phi:A \to F\{ \tau \}$, $n \mapsto \Phi_n$, where the image of $T$ is 
$$
\Phi_T = T + c_1 \tau + \cdots + c_r \tau^r,
$$ 
with $c_i \in F$, $c_r \neq 0$, and $r \geq 1$. 

For all but finitely many monic irreducibles $p$ in $A$, the reduction of $\Phi$ modulo $p$ is a Drinfeld module over the residue field of $p$ of rank $r$, and the Frobenius endomorphism at $p$ satisfies a polynomial of degree $r$, which we refer to as the characteristic polynomial of Frobenius at $p$ for $\Phi$. This polynomial has coefficients in $A$, and (additive inverse of) the next-to-leading coefficient of this polynomial $a_p(\Phi)$ gives rise to interesting arithmetic questions. For instance, one can fix an element $a \in A$ and ask for the distribution of primes $p$ for which $a_p(\Phi)=a$.

Interest in this question can be traced back to the analogous question for elliptic curves over $\Q$. In \cite{LangTrotter1976}, Lang and Trotter give the following conjecture.

\begin{conjecture}
Let $E$ be an elliptic curve over $\Q$ without complex multiplication, and fix an integer $a$. For $\ell \in \N$ a prime of good reduction, let $a_\ell$ be the trace of Frobenius at $\ell$ of $E$. Then,
$$
\# \{
\ell \texteqn{prime in $\Z$:} 1 \leq \ell \leq x, a_{\ell}(E) = a
\}
\sim
C_{E,a}\frac{\sqrt{x}}{\log x},
$$
where $C_{E,a}$ is a constant depending only on $E$ and $a$.
\end{conjecture}

The conjecture remains unproven, but progress has been made in the form of upper bounds (see for instance \cite{Serre1981}, \cite{CojocaruFouvryMurty2005}, and \cite{MurtyMurtySaradha1988}). 

An approximation to this problem is to consider an average over a ``box'' of elliptic curves. Probably the first such computation can be found in the work of Fouvry and Murty \cite{FouvryMurty1996}, where they compute the average number of primes for which $a_{\ell}=0$ for the two-parameter family of elliptic curves $ \{ E_{(a,b)}:y^2=x^2+ax+b \}$. They obtain the following asymptotic formula (the dash indicates a sum over models of non-singular curves):
$$
\frac{1}{4AB}
\sum_{\substack{|a| \leq A \\ |b| \leq B}} {}^{'}
\# \{
\ell \texteqn{prime in $\Z$:} 1 \leq \ell \leq x, a_{\ell}\left(E_{(a,b)}\right) = 0
\}
\sim
\frac{\pi}{3} \frac{\sqrt{x}}{\log x},
$$
provided $A,B > x^{\frac{1}{2} + \e}$ and $AB > x^{\frac{3}{2} + \e}$ for some $\e > 0$. More generally, one has \cite{DavidPappalardi1999}, where David and Pappalardi fix an integer $a$ and compute the average number of primes for which $a_{\ell}=a$, obtaining
$$
\frac{1}{4AB}
\sum_{\substack{|a| \leq A \\ |b| \leq B}} {}^{'}
\# \{
\ell \texteqn{prime in $\Z$:} 1 \leq \ell \leq x, a_{\ell}\left(E_{(a,b)}\right) = a
\}
\sim
C_{a} \frac{\sqrt{x}}{\log x},
$$
with the constant $C_{a}$ given explicitly, provided $A,B > x^{1 + \e}$ for some $\e > 0$. Improvements to the result in \cite{DavidPappalardi1999} come in the form of smaller box sizes obtained through the use of character sum estimates (see for instance \cite{Baier2007} and \cite{BanksShparlinkski2009}). For instance, Baier \cite{Baier2007} reduces the conditions on $A$ and $B$ to $A,B > x^{\frac{1}{2} + \e}$ and $AB > x^{\frac{3}{2} + \e}$ for some $\e > 0$.

It is natural to ask for average results such as these for rank 2 Drinfeld modules. To this effect, David \cite{David1996} obtains an analogue of the result of Fouvry and Murty, counting supersingular reductions for rank 2 Drinfeld modules on average. More precisely, as $x$ runs through positive integers and $x \to \infty$, 

$$
\frac{1}{q^{\mc A + 1}q^{\mc B + 1}}
\sum_{\substack{\deg g \leq \mc A \\ \deg \gD \leq \mc B}} {}^{'}
\# \{
\ell \in A \texteqn{monic irreducible,} \deg \ell = x: a_{\ell}(E_{(a,b)}) = 0
\}
\sim
C(x,q) \frac{q^{x/2}}{x},
$$

with the constant $C(x,q)$ given explicitly in terms of $q$ and the parity of $x$, and provided $\mc A, \mc B \geq x$.

In this paper we prove a generalization of this last statement for more general $a \in A$. Our character sum estimates require us to avoid very small values of $q$, say $q \geq 17$. Our techniques follow those of \cite{DavidPappalardi1999} and \cite{Baier2007}, giving us lower bounds on $\mc A$ and $\mc B$ that improve as $q \to \infty$. 

\subsection{Statement of Results}

For $g, \gD \in A$, let $\Phi(g,\gD)$ be the Drinfeld module over $F$ determined by the homomorphism
$$
\Phi(g,\gD):A \to F\{\tau\},\hspace{16pt} T \mapsto T + g \tau + \gD \tau^2,
$$
and for a prime $p$ of good reduction, let $a_p(g, \gD)$ and $P_p(X,g,\gD)$ be (respectively) the trace and the characteristic polynomial of the Frobenius at $p$ for $\Phi(g,\gD)$.

Fix a positive integer $x$, and let $\mc P$ be the set of monic irreducible polynomials in $A$ of degree $x$. Let $\mc A, \mc B$ be positive integers; these will be constrained in terms of $x$. We will take an average over Drinfeld modules in a two-parameter family by taking
$$
\Box(\mc A, \mc B) := 
\{ 
(g, \gD) \in A \times A: \deg g < \mc A, \deg \gD < \mc B, \Phi(g, \gD) \text{  is a rank 2 Drinfeld module} 
\}.
$$

We are interested in the quantity
$$
S(x, \mc A,\mc B, a, u) :=
\frac{1}{\# \Box(\mc A , \mc B)}
\sum_{(g, \gD) \in \Box(\mc A, \mc B)}
\# \{ p \in \mc P: P_p(X,g,\gD) = X^2 -aX+up \}.
$$

Define $C_\infty$, which we view as the local factor at the prime at infinity, as
\begin{equation}\label{eqn:FactorAtInfinity}
C_{\infty} :=
\left\{\begin{array}{ll}
\frac{1}{q^{1/2}(q-1)} & \texteqn{if}x \texteqn{odd,}\\ 
\frac{1}{(q+1)(q-1)} & \texteqn{if}x \texteqn{even}. 
\end{array}\right.
\end{equation}

Our main theorem makes use of character sum estimates to allow for smaller values of $\mc A$ and $\mc B$ for large $q$. 

\begin{theorem} \label{thm:MainTheoremSmallBox}
Take the notation above, and assume $q \geq 17$ is fixed. Suppose that $\deg a < \frac{1}{2}x$ and that,

\begin{equation}\label{eqn:RestrictionsInMainTheorem}
\mc A, \mc B >  \frac{\log 4}{\log q}x + \log x \; \; \mathrm{and} \; \;
\mc A + \mc B > \left( \frac{1}{2} +  \frac{\log 16}{\log q} \right) x + \log x.
\end{equation}
Then, we have
$$
S(x, \mc A,\mc B, a, u)
= 
C_{\infty}
C(a)
\frac{q^{x/2}}{x}
+
E(x,q);
$$
where $C_{\infty}$ is defined in (\ref{eqn:FactorAtInfinity}), $C(a)$ is given by
$$
\prod_{\ell \mid a}
\left(
1 - \frac{1}{|\ell|^2}
\right) ^{-1}
\prod_{\ell \nmid a}
\left(
\frac{|\ell|(|\ell|^2 - |\ell|-1)}{(|\ell|^2-1)(|\ell|-1)}
\right),
$$
and $E(q,x)$ is $\ul o \left( \frac{q^{x/2}}{x} \right)$ as $x \to \infty$.
\end{theorem}

In the setting studied by \cite{DavidPappalardi1999}, there is no analogue of the unit appearing in the definition of $S(x, \mc A,\mc B, a, u)$; our result shows that, on average, there is uniform distribution over admissible $u \in \F^*$. The box size we obtain is comparable to \cite{Baier2007}, with the difference that, for $q$ large enough, the factors of $x$ appearing in the expressions in (\ref{eqn:RestrictionsInMainTheorem}) can be made arbitrarily close to $0$ and $\frac{1}{2}$ respectively.

\subsection{Notation}

Let $q$ be an odd prime power, and let $\F$ be the finite field with $q$ elements. Let $A:=\F[T]$, $\upone A$ the set of monic polynomials in $A$, and $F :=\F(T)$. For an element $a \in A$, let $\sgn{a}$ denote the leading coefficient of $a$ and $\deg{a}$ denote the degree of $a$ as a polynomial in $T$. Write $|a| = q^{\deg{a}}$.

The letters $p$ and $\ell$ hereafter denote monic irreducible elements of $A$, which we refer to as primes. If $p$ is a prime in $A$, let $\F_p$ be the $A$-field $A / pA$. If $a \in A$, then we write $\hat a$ for the element $a \mod p$ in $\F_p$.

Define $\varphi(\cdot)$ to be the Euler-phi function on $A$, i.e. $\varphi(a) := \#(A/aA)^*$. As in the classical case, $\varphi$ is multiplicative and has the product expansion
$$
\varphi(a)=|a|\prod_{p \mid a}\left(1 - \frac{1}{|p|}\right).
$$
From this expansion it is straightforward to deduce the identity
\begin{equation}\label{eqn:ProductPropertyPhi}
\varphi(vw) = \varphi(v) \varphi(w)\cdot\frac{\gcd(v,w)}{\varphi(\gcd(v,w))}. 
\end{equation}

Fix a positive integer $x$, and write $\mc P := \{ p \in \upone{A}: \deg p = x\}$.

\section{Background and Preliminary Lemmas}

\subsection{Drinfeld modules}

We refer the reader to \cite{Drinfeld1974}, \cite{Drinfeld1977}, \cite{Gekeler1991}, and \cite{Rosen2002} for more rigorous treatments of the theory of Drinfeld modules, including proofs.

For any $A$-field $K$, let $K \{ \tau \}$ be the non-commutative polynomial ring with the commutation rule $\tau \ga = \ga^q \tau $ for all $\ga \in K$. 

 For $(g,\gD) \in A \times A$ let $\Phi(g,\gD)$ be the Drinfeld module over $F$ determined by the homomorphism
$$
\Phi(g,\gD):A \to F\{\tau\},\hspace{16pt} T \mapsto T + g \tau + \gD \tau^2.
$$

If $\gD \neq 0$, this is a rank 2 Drinfeld module. For any prime $p \nmid \gD$, let

$$
P_p(X,g,\gD) = X^2 - a_p(g,\gD) X + u_p(g,\gD)p \in A[X]
$$

be the characteristic polynomial of the Frobenius endomorphism at $p$; recall from \cite{Gekeler1991} that this polynomial satisfies $\deg a_p(g, \gD) \leq \frac{1}{2}\deg p$ and $u_p(g, \gD)$ is a unit in $A$.

For $(\gamma, \gd) \in \F_p \times \F_p$ with , let $\phi(\gamma, \gd)$ be the Drinfeld module over $\F_p$ determined by 
$$
\phi(\gamma,\gd):A \to \F_p \{\tau\},\hspace{16pt} T \mapsto \hat{T} + \gamma \tau + \gd \tau^2.
$$

If $\gd \neq 0$, this is a rank two finite Drinfeld module over $\F_p$, and we write
$$
P(X,\gamma, \gd) = X^2 - a(\gamma, \gd) X + u(\gamma, \gd)p
$$
for the characteristic polynomial of the Frobenius endomorphism.

\subsection{Orders in imaginary quadratic fields}

We refer the reader to \cite{Yu1995}, \cite{Gekeler2008} for more details and proofs.

Let $E$ be a degree 2 extension of $F$; since we assume $q$ is odd, $E$ must be separable. We say $E$ is an imaginary quadratic extension of $F$ if the prime at infinity of $F$ does not split in $E$. An equivalent, more concrete characterization (see for instance \cite[Prop. 14.6, p. 248]{Rosen2002}) is given by writing $E = F(\sqrt{D})$ with $D$ squarefree. Then, $E$ is an imaginary quadratic extension if and only if $\deg D$ is odd, or $\deg D$ is even and $\sgn D$ is a nonsquare in $\F^*$. Whenever $D \in A$ satisfies the above, we say $D$ is a fundamental discriminant. If $d$ satisfies the above except possibly the squarefree condition, we simply say that $d$ is a discriminant.

Fix a fundamental discriminant $D$ and let $E = F(\sqrt{D})$. Then $\mc O_E := A\left[\sqrt D\right]$ is the ring of integers of $E$. Moreover, any $A$-subalgebra of $E$ with $A$-rank 2 is contained in $\mc O_E$; we cal these orders in $E$, with $\mc O_E$ being the maximal order in $E$. As in the classical setting, there is a one-to-one correspondence between discriminants (up to multiplication by a square in $\F^*$) and orders in imaginary quadratic extensions of $F$, given by $d \leftrightarrow A\left[\sqrt d\right]$. This correspondence sends fundamental discriminants to maximal orders.

For each fundamental discriminant $D$, we define the Dirichlet character $\chi_D$ on $A$ associated to $A[\sqrt D]$ by assigning values at primes via
$$
\chi_D(\ell):=
\left\{\begin{array}{ll}
0 & \texteqn{if}\ell \texteqn{is ramified in $F(\sqrt D)$,}\\ 
1 & \texteqn{if}\ell \texteqn{splits in $F(\sqrt D)$,}\\ 
-1 & \texteqn{if}\ell \texteqn{is inert in $F(\sqrt D)$,}
\end{array}\right.
$$
and extending completely multiplicatively to all of $A$. For each discriminant $d$, write $d = f^2D$ where $D$ is a fundamental discriminant and $f \in \upone A$. Then, we define the Dirichlet character $\chi_d$ on $A$ associated to $A[\sqrt d]$ by assigning values at primes via
$$
\chi_d(\ell):=
\left\{\begin{array}{ll}
0 & \texteqn{if}\ell \mid f,\\ 
\chi_D(\ell) & \texteqn{otherwise}
\end{array}\right.
$$
and extending completely multiplicatively to all of $A$. For the Dirichlet character described above, define the $L$-function associated to $\chi_d$ via 
$$
L(s, \chi_d) := \sum_{n \in \upone A} \frac{\chi_d(n)}{|n|^s} \texteqn{for} \Re(s) > 1.
$$

One has (see for instance \cite[Ch. 2]{Rosen2002}) that this function can be expressed as a polynomial in $q^{-s}$ of degree $\deg d$. Moreover, the Riemann Hypothesis for function fields implies that, replacing $q^{-s}$ with $X$, we can write the $L$-function of a Dirichlet character $\chi$ on $A$ whose conductor has degree $x$ as
\begin{equation} \label{eqn:LfunctionEulerProduct}
L(s, \chi) = \prod_{i=1}^{x} \left( 1 - \ga_i X \right)
\texteqn{with} \ga_i \in \C, |\ga_i| \leq q^{1/2}.
\end{equation}

We also define the class number of the order $A\left[\sqrt d\right]$. First we define a proper ideal $I \subset A\left[\sqrt d\right]$ as one that satisfies
$$
\{ \ga \in E: \ga I \subseteq I \} = A\left[\sqrt d\right].
$$

The proper ideals of $A\left[\sqrt d\right]$ have unique factorization into products of proper prime ideals, and we define $h(d)$, the ideal class number of $A\left[\sqrt d\right]$, as the number of proper $A\left[\sqrt d\right]$-ideals modulo principal ideals. From the theory of quadratic spaces (see \cite{Yu1995}), we get that, if $D$ is a fundamental discriminant and $f \in A$ the ratio of the class numbers $h(f^2D)$ and $h(D)$ is

\begin{equation}\label{eqn:ClassNumberRatio}
\frac{h(f^2D)}{h(D)} 
= 
\frac
{|f|}
{ 
  \left[ 
  A\left[\sqrt D\right]^*:A\left[\sqrt {f^2 D}\right]^*   
  \right]    
}
\prod_{\ell \mid f} 
\left(1 - \frac{\chi_D(\ell)}{|\ell|}  \right).
\end{equation}

For maximal orders, we have the analytic class formula from Artin's thesis (\cite{Artin1924}, see also \cite[Theorem 17.8A, p. 317]{Rosen2002}):

\begin{equation}\label{eqn:ClassNumberFormulaFundamental}
h(D)
=
\left\{\begin{array}{ll}
\frac{|D|^{1/2}}{q^{1/2}}L(1,\chi_D) & \texteqn{if}\deg D \texteqn{odd,}\\ 
\frac{2|D|^{1/2}}{q+1}L(1,\chi_D) & \texteqn{if}\deg D \texteqn{even and $\sgn D$ is a nonsquare in $\F^*$.} 
\end{array}\right.
\end{equation}

We combine (\ref{eqn:ClassNumberRatio}) and (\ref{eqn:ClassNumberFormulaFundamental}) in the following lemma for later reference.

\begin{lemma}\label{lem:ClassNumberFormulaGeneral}
Take the notation above and suppose that $d=f^2D$, where $D$ is a fundamental discriminant and $f$ is monic. Then,
$$
h(d) = 
\left\{\begin{array}{ll}
\frac{|d|^{1/2}}{q^{1/2}}L(1,\chi_d) & \texteqn{if}\deg d \texteqn{odd,}\\ 
\frac{2|d|^{1/2}}{q+1}L(1,\chi_d) & \texteqn{if}\deg d \texteqn{even and $\sgn d$ is a nonsquare in $\F^*$.} 
\end{array}\right.
$$
\end{lemma}

\subsection{Finite Drinfeld modules of rank 2}
In this section we record some lemmas regarding isomorphism classes of finite Drinfeld modules. Let $p$ be a prime in $A$ with $\deg p = x$. 

\begin{theorem}[{\cite[Prop. 6.8]{Gekeler2008}}]\label{thm:DeuringGekelerFormula}
Fix $a \in A$ with $\deg a < x/2$. Let $u \in \F^*$, and if $x$ is even then further assume that $-4u$ is a nonsquare in $\F^*$. The number of $\F_p$-isomorphism classes of rank 2 Drinfeld modules with characteristic polynomial of Frobenius equal to $X^2 - aX + up$ is the total number of ideal classes of the ring $A\left[\sqrt{a^2 - 4up}\right]$.
\end{theorem}

We denote this quantity as $H_{p}$, and observe that
\begin{equation}\label{eqn:GaussCNtoUsualCN}
H_{p} = 
\sum_{\substack{f \in \upone A \\ f^2 \mid a^2 - 4up}}
h\left( \frac{a^2 - 4up}{f^2} \right).
\end{equation}

\begin{lemma}[{\cite[Section 1]{Gekeler2008}}]\label{lem:DescribeIsoClassesofFiniteDM}
Let $\phi(\ga, \gb)$ and $\phi(\gamma, \gd)$ be rank 2 Drinfeld modules over $\F_p$ (i.e. $\gb \gd \neq 0$).
Then, $\phi(\ga, \gb)$ and $\phi(\gamma, \gd)$ are $\F_p$-isomorphic if and only if there exists $\mu \in \F_p^*$ such that
$$
(\ga, \gb) = (\mu^{q-1}\gamma, \mu^{q^2-1}\gd).
$$
Moreover,
$$
\# \{
(\ga, \gb): \phi(\ga, \gb) \equiv_{\F_p} \phi(\gamma, \gd)
\}
=
\left\{\begin{array}{ll}
\frac{|p|-1}{q^2-1} & \texteqn{if} \gamma = 0 \texteqn{and} \deg p \texteqn{even,}\\ 
\frac{|p|-1}{q-1} & \texteqn{otherwise}.\\ 
\end{array}\right.
$$
\end{lemma}

\begin{lemma}\label{lem:CongruenceConditionsforCharSums}
Let $a,b,c,d \in A$ with $p \nmid abcd$. Let $\Phi(a,b)$ and $\Phi(c,d)$ be rank 2 Drinfeld modules over $F$. Then, the following are equivalent:
\begin{itemize}
\item[1.] The finite Drinfeld modules over $\F_p$ given by $\phi(\hat a,\hat b)$ and $\phi(\hat c,\hat d)$ are $\F_p$-isomorphic.
\item[2.] $\hat c \hat a^{-1}$ is a perfect $(q-1)^{\text{th}}$ power in $\F_p^*$, and $(ca^{-1})^{(q+1)} \equiv db^{-1} \mod p$
\end{itemize}
\end{lemma}

\begin{proof}
This follows directly from \cite[\S 1.4]{Gekeler2008}
\end{proof}

\begin{lemma}\label{lem:CountUnusualIsomorphismClasses}
The number of $\F_p$-isomorphism classes of rank 2 Drinfeld modules over $\F_p$ containing a representative of the form $\phi(0,\gd)$ for some $\gd \in \F_p$ is $O(q^2)$.
\end{lemma}

\begin{proof}
The pairs $(0, \gd)$ can be parametrized by $\F_p^*$, and multiplication by a $(q^2-1)^{\text{th}}$ power gives $\F_p$-isomorphic Drinfeld modules. Therefore our quantity is bounded by $$ \# \left( \F_p^* / (\F_p^*)^{q^2-1}   \right) = q^2 - 1.$$
\end{proof}

\subsection{Character sum estimates for Dirichlet characters over $A$}

The first lemma is a straightforward consequence of the Riemann Hypothesis for function fields. We record it, with proof, for the sake of completeness. In particular, we later make use of the explicit dependence of the bound on the degree of the conductor of the character.

\begin{lemma}\label{lem:CharSumOnRH}
Let $p \in \upone A$ be a prime with $\deg p = x$. Fix natural numbers $z' \leq z \leq x$. Fix a nonprincipal Dirichlet character $\chi$ modulo $p$. Then we have the upper bound
$$
\left|
\sum_{\substack{f \in \upone A \\ z' \leq \deg f \leq z}}
     \chi(f)
\right|
\leq
q^{z/2} 2^x.
$$
\end{lemma}

\begin{proof}
We write the $L$-function $L(s,\chi)$ associated to $\chi$ in two different ways and compare coefficients. By definition, for $\Re s > 1$ we have
$$
L(s,\chi) = \sum_{f \in \upone A} \chi(f) |f|^{-s}.
$$
Under the change of variables $X \leftrightarrow q^{-s}$, and writing $$c(k) := \sum_{f \in \upone A, \deg f =k } \chi(f),$$ this becomes
$$
L(s,\chi) = \sum_{k=1}^x c(k)X^k.
$$
Examining the product expansion given in (\ref{eqn:LfunctionEulerProduct}), we see that $c(k)$ is the $k^{\text{th}}$ elementary symmetric polynomial in $x$ letters evaluated at $\{ -\ga_i \}_{i=1}^x$. Ignoring cancellation among summands in this presentation of $c(k)$, we get the upper bound
$$
|c(k)| \leq \binom{x}{k} q^{k/2},
$$
which immediately yields the bound
$$
\left|
\sum_{\substack{f \in \upone A \\ z' \leq \deg f \leq z}}
     \chi(f)
\right|
\leq
q^{z/2}
\sum_{k=z'}^z \binom{x}{k}.
$$

The lemma follows by completing the sum of binomial coefficients to range from 0 to $x$.
\end{proof}

\begin{lemma}\label{lem:OrthogonalityCharSumSquare}
Let $p \in \upone A$ be a prime. Fix a natural number $z$.
Let $\{a_n\}$ be a sequence of complex numbers supported on $\{n \in \upone A: \deg n \leq z \}$. We have the bound
$$
\sum_{\chi(\mod p)}
\left |
\sum_{\substack{n \in \upone A \\ \deg n \leq z }}
a_n \chi(n)
\right| ^2
=
\varphi(p)
\sum_{\deg f < \deg p}
\left |
\sum_{\substack{n \in \upone A \\ \deg n \leq z \\ n \equiv f \mod p }}
a_n
\right| ^2,
$$
where the sum is taken of all Dirichlet characters mod $p$.
\end{lemma}

\begin{proof}
Expand the square on the left hand side and use orthogonality of characters. This is essentially the same as \cite[Lemma 5]{Baier2007}.
\end{proof}

\section{From Drinfeld modules to a sum of class numbers}
Our first task is to remove the dependence on the box size from the calculation; our approach follows \cite{Baier2007} and uses character sums to simultaneously treat congruence conditions and isomorphism conditions. 

From here on, we fix $a \in A$ and $u \in \F^*$. We assume that $x > 2 \deg a$; the case of equality poses additional complications that would require further study to address. If $x$ is even, then we further assume that $-4u$ is a nonsquare in $\F^*$. We define $I_{p}$ to be the number of $\F_p$ isomorphism classes  $[\phi]$ of rank 2 Drinfeld modules over $\F_p$ such that:
\begin{itemize}
\item[1.] For any representative $\phi(\gamma, \gd)$ of $[\phi]$, $P(\gamma, \gd) = X^2 - aX +up$; and
\item[2.] There does not exist a representative of $[\phi]$ of the form $\phi(0,\gd)$.
\end{itemize}

By Lemma \ref{lem:CountUnusualIsomorphismClasses}, we have $H_{p} - I_{p} = O(q^2)$.

\subsection{Application of character sums estimates}\label{subs:SmallBoxes}
We change the order of summation to obtain
\begin{align*}
S(x, \mc A, \mc B, a, u)
=
\frac{1}{\# \Box(\mc A, \mc B)}
\sum_{p \in \mc P}
  \# \{
    (g, \gD) \in \Box(\mc A, \mc B): p \nmid \gD, P_p(X,g,\gD) = X^2 -aX + up
  \} \\
=
\frac{1}{\# \Box(\mc A, \mc B)}
\sum_{p \in \mc P}
  \# \{
    (g, \gD) \in \Box(\mc A, \mc B): p \nmid g \gD, P_p(X,g,\gD) = X^2 -aX + up
  \} 
  + O \left( \frac{q^{\mc A}}{|p|} \right).
\end{align*}

Let $(u_{p,j}, v_{p,j})$ for $j = 1 , .. , I_{p}$ be pairs of monic elements of $A$ such that 
$$
\left\{
\phi(\hat u_{p,j}, \hat v_{p,j})
\right\}_{j=1}^{I_{p}}
$$
forms a system of representatives of $\F_p$-isomorphism classes of rank 2 Drinfeld modules over $\F_p$ with $P(X,\hat u_{p,j}, \hat v_{p,j})= X^2 -aX + up$ and $p \nmid u_{p,j}$. With this notation the expression inside of this last sum becomes
\begin{equation} \label{eqn:InnerTermBeforeCharacters}
\sum_{j=1}^{I_{p,a}}
  \# \{
    (g, \gD) \in \Box(\mc A, \mc B): \phi(\hat g, \hat \gD) \cong_{\F_p} \phi(\hat u_{p,j}, \hat v_{p,j})
    \} 
\end{equation}

By Lemma \ref{lem:CongruenceConditionsforCharSums}, $(g, \gD)$ satisfies the condition in the summand of (\ref{eqn:InnerTermBeforeCharacters}) if and only if:
\begin{itemize}
\item[1.] $g u_{p,j}^{-1}$ is a perfect $(q-1)^{\text{th}}$ power modulo $p$, and
\item[2.] $(g u_{p,j}^{-1})^{q+1} \equiv \gD v_{p,j}^{-1} \pmod p$.
\end{itemize}

To detect the first condition, we use the $(q-1)^{\text{th}}$ power residue symbol $\legendre{\cdot}{\cdot}_{(q-1)}$. It has properties similar to those of power residue symbols in the classical setting; see \cite[Chapter 3]{Rosen2002} for details. In particular, we have
$$
\sum_{k=0}^{q-2}
\legendre{gu_{p,j}^{-1}}{p}^k_{(q-1)}
=
\left\{\begin{array}{ll}
q-1 & \texteqn{if $g u_{p,j}^{-1}$ is a perfect $(q-1)^{\text{th}}$ power modulo $p$,}\\ 
0 & \texteqn{otherwise.}
\end{array}\right.
$$

For the second condition, orthogonality of Dirichlet characters implies that
$$
\sum_{\chi (\mod p)}
\chi( g^{q+1} u_{p,j}^{-(q+1)} \gD^{-1} v_{p,j})
=
\left\{\begin{array}{ll}
\varphi(p) & \texteqn{if}(g u_{p,j}^{-1})^{q+1} \equiv \gD v_{p,j}^{-1} \pmod p,\\ 
0 & \texteqn{otherwise.}
\end{array}\right.
$$

We now write (\ref{eqn:InnerTermBeforeCharacters}) as
$$
\frac{1}{(q-1)\varphi(p)} 
\sum_{j=1}^{I_{p}} \sum_{(g, \gD) \in \Box(\mc A, \mc B)}
  \left(
    \sum_{k=0}^{q-2}   \legendre{gu_{p,j}^{-1}}{p}^k_{(q-1)}
  \right)
  \left(
    \sum_{\chi (\mod p)}  \chi( g^{q+1} u_{p,j}^{-(q+1)} \gD^{-1} v_{p,j})
  \right).
$$

For $\chi$ a Dirichlet character mod $p$, and $0 \leq k \leq q-2$, define
\begin{align*}
&G_1(\chi,k) 
:= 
\sum_{j=1}^{I_{p}}
  \legendre{u_{p,j}^{-1}}{p}^k_{(q-1)}
  \chi^{q+1}(u_{p,j}^{-1}) \chi(v_{p,j}), \\
&G_2(\chi,k)
:=
\sum_{\substack{g \in \upone A \\ \deg g < \mc A}}
  \legendre{g}{p}^k_{(q-1)}
  \chi^{q+1}(g), \texteqn{and} \\
&G_3(\chi,k)
:=
\sum_{\substack{\gD \in \upone A \\ \deg \gD < \mc B}}
  \chi(\gD).
\end{align*}

In this notation, (\ref{eqn:InnerTermBeforeCharacters}) becomes
\begin{equation}\label{eqn:RewriteWithCharacterSums}
\frac{1}{(q-1)\varphi(p)} 
\sum_{k=0}^{q-2} \sum_{\chi (\mod p)} 
G_1(\chi,k) G_2(\chi,k) G_3(\chi,k). 
\end{equation}

In (\ref{eqn:RewriteWithCharacterSums}), there is exactly one totally trivial term, namely $(\chi,k) = (\chi_0,0)$. The contribution of this term is
$$
\frac{I_{p}q^{\mc A + \mc B}}{(q-1) \varphi(p)}
=
\frac{H_{p}q^{\mc A + \mc B}}{(q-1) |p|} 
\left( 1 + O \left(\frac{1}{|p|} \right) \right)
+
O \left(\frac{q^{2+\mc A + \mc B}}{(q-1) \varphi(p)} \right).
$$

There are two types of semi-trivial terms in (\ref{eqn:RewriteWithCharacterSums}). Suppose first that $\chi \neq \chi_0$ and $\legendre{\cdot}{\cdot}^k_{(q-1)}\chi^{q+1} = \chi_0$. This means that $\chi^{q+1} = \legendre{\cdot}{\cdot}^{-k}_{(q-1)}$, and there are at most $q+1$ choices of $\chi$ satisfying this relation. Altogether, there are at most $O(q^2)$ such terms overall. For each of these terms, we estimate $G_1$ and $G_2$ trivially, and bound $G_3$ using Lemma \ref{lem:CharSumOnRH}. The total contribution of these terms is 
$$
O \left(
\frac{I_{p} 2^x q^{2+\mc A + (\mc B / 2)}}{(q-1)|p|} 
\right).
$$

Now suppose that $\chi = \chi_0$ and $k \neq 0$, which occurs for $O(q)$ terms. Here we estimate $G_1$ and $G_3$ trivially, and bound $G_2$ using Lemma \ref{lem:CharSumOnRH}. The total contribution of these terms is
$$
O \left(
\frac{I_{p} 2^x q^{1+(\mc A /2)+ \mc B}}{(q-1)|p|} 
\right).
$$

For the remaining terms in (\ref{eqn:RewriteWithCharacterSums}), we apply the Cauchy-Schwarz inequality to obtain
\begin{equation}\label{eqn:AfterCauchySchwartz}
\begin{aligned}
\frac{1}{(q-1)\varphi(p)} 
&\sum_{k=1}^{q-2} 
\left(
  \sum_{
    \substack{
      \chi (\mod p), \chi \neq \chi_0 
      \\ 
      \chi^{q+1} \neq \legendre{\cdot}{\cdot}^{-k}_{(q-1)}
    }
  } 
G_1(\chi,k) G_2(\chi,k) G_3(\chi,k)
\right) \ll  \\
&\sum_{k=1}^{q-2} 
  \left(
    \sum_{\chi(\mod p)} 
    \left|
      \sum_{j=1}^{I_{p}}
        \legendre{u_{p,j}^{-1}}{p}^k_{(q-1)}
        \chi(u_{p,j}^{-(q+1)}v_{p,j})
    \right|^2
  \right)^{1/2} \times \\
  &\left(
    \sum_{\chi (\mod p), \chi \neq \chi_0}
    q
    \left|
     \sum_{\substack{g \in \upone A \\ \deg g < \mc A}} 
       \chi(g)
    \right|^4
  \right)^{1/4}
  \left(
    \sum_{\chi (\mod p), \chi \neq \chi_0}
        \left|
          \sum_{\substack{\gD \in \upone A \\ \deg \gD < \mc B}} 
            \chi(\gD)
        \right|^4
  \right)^{1/4}
\end{aligned}
\end{equation}

The $q$ in the second factor on the right hand side of (\ref{eqn:AfterCauchySchwartz}) appears since, for fixed $\chi_1$, there are $O(q)$ choices of $\chi$ satisfying $\chi^{q+1} = \chi_1$. The second and third factors on the right hand side of (\ref{eqn:AfterCauchySchwartz}) are bounded using Lemma \ref{lem:CharSumOnRH}. For the first factor on the right hand side of (\ref{eqn:AfterCauchySchwartz}), we apply Lemma \ref{lem:OrthogonalityCharSumSquare}, noting that
$$
\# \{ j: u_{p,j}^{-(q+1)}v_{p,j} \equiv f \pmod p \}
\leq q
$$
for all $f \mod p$, and is nonzero for at most $I_{p}$ choices of $f \mod p$. Altogether, the contribution from these terms is bounded by
$$
O \left(
\frac{|p|^{1/2} I_{p}^{1/2}4^x q^{(3 + \mc A + \mc B) / 2}}{(q-1) |p|}
\right).
$$

Combining the above estimates, we have
$$
S(x, \mc A, \mc B, a,u) =
\frac{q^{\mc A + \mc B}}{\#\Box(\mc A, \mc B)} \frac{1}{(q-1)q^x}
\sum_{p \in \mc P} H_{p}
+
E,
$$
where
$$
E
=
O \left(
    \frac{1}{x q^{\mc B}}
    + 
    \frac{1}{q^{2x+1}}
    \sum_{p \in \mc P} H_{p}
    +
    \frac{2^x}{q^{x+1}}
    \left(
    \frac{q}{q^{\mc A / 2}} + \frac{q^2}{q^{\mc B / 2}}
    \right)
    \sum_{p \in \mc P} H_{p}
    +
    \frac{q^{1/2}}{q^{x/2}}
    \frac{4^x}{q^{(\mc A + \mc B)/2}}
    \sum_{p \in \mc P} (H_{p})^{1/2}.
\right)
$$
From the analytic class formula and a trivial bound on $L(1, \chi)$, we can write
$$
H_{p} \ll x^2q^{x/2} \texteqn{and} (H_{p})^{1/2} \ll xq^{x/4};
$$
this allows us to rewrite $E$ as
$$
E 
=
O \left(
  1 
  +
  \frac{q^{x/2}}{(q-1)x}
  \left(
    \frac{x^2}{q^{2x+1}}
    +
    x^2 2^x 
    \left(
      \frac{q}{q^{\mc A / 2}} + \frac{q^2}{q^{\mc B / 2}} 
    \right)
    +
    \frac{4^x q^{3/2 + x/4}}{q^{\mc A / 2+\mc B / 2}}
  \right)
\right)
$$

In order for the main term to dominate $E$ as $x \to \infty$, we require that $\mc A$ and $\mc B$ satisfy (\ref{eqn:RestrictionsInMainTheorem}).

\section{Application of the prime number theroem for arithmetic progressions}

We now focus on the quantity
\begin{equation}\label{eqn:SumofClassNumbers}
\frac{1}{(q-1)q^x} \sum_{p \in \mc P} H_{p},
\end{equation}

upon observing that
$$
\frac{q^{\mc A + \mc B}}{\#\Box(\mc A, \mc B)} = 1+\ul o (1) \texteqn{as} q^x \to \infty.
$$

Combining the analogue of Deuring's lemma (Theorem \ref{thm:DeuringGekelerFormula}) with the class number formula (Lemma \ref{lem:ClassNumberFormulaGeneral}), and taking $C_\infty$ as defined in (\ref{eqn:FactorAtInfinity}), we get
$$
H_{p} 
= 
(q-1) C_\infty 
\sum_{\substack{r \in \upone A \\ r^2 | a^2 - 4up}}
\frac{|p|^{1/2}}{|r|} L(1,\chi_R),
$$
where $R := R(r,u,a,p) = \frac{a^2 - 4up}{r^2}$. Altogether we rewrite (\ref{eqn:SumofClassNumbers}) as
\begin{equation}\label{eqn:SumofLValues}
\frac{C_\infty}{q^{x/2}}
\sum_{p \in \mc P}
\sum_{\substack{r \in \upone A \\ r^2 | a^2 - 4up}}
\frac{L(1,\chi_R)}{|r|}.
\end{equation}

We recall an effective version of the prime number theorem for arithmetic progressions. For a proof, see \cite[p.40-42]{Rosen2002}, while keeping track of the dependence of the $O$-constant on $m$.

\begin{theorem}\label{thm:Effective PNTAP}
Suppose $m,a \in A$ with $(m,a) = 1$. Then
$$
\# \left\{ p, \deg p = x: p \equiv a \mod m \right\}
=
\frac{1}{\varphi(m)}
\frac{q^x}{x}
+
O \left(
  \frac{(\deg m)^2}{x} q^{x/2}
\right).
$$
\end{theorem}

\subsection{Sum truncations}
We follow the strategy of \cite{DavidPappalardi1999}, that is, upon expanding the $L$-value we get a triple sum, truncate in two directions, and control the third direction using Theorem \ref{thm:Effective PNTAP}.

Expressing the $L$-value as a sum, we rewrite (\ref{eqn:SumofLValues}) as
\begin{equation}
\frac{C_\infty}{q^{x/2}}
\sum_{p \in \mc P}
\sum_{\substack{r \in \upone A \\ r^2 | a^2 - 4up}}
\frac{1}{|r|}
\sum_{\substack{v \in \upone A }}
\frac{\chi_R(v)}{|v|}.
\end{equation}

We fix a parameter $U$, depending on $x$, to be determined precisely later. We aim to control the contribution of $r \in \upone A$ with $\deg r \geq U$. Changing the order of summation and bounding the $L$-value trivially, we get that this is
$$
\frac{C_\infty}{q^{x/2}}
\sum_{p \in \mc P}
  \sum_{\substack{r \in \upone A \\ r^2 | a^2 - 4up \\ \deg r \geq U}}
    \frac{1}{|r|}
      \sum_{\substack{v \in \upone A}}
        \frac{\chi_R(v)}{|v|}
=
O \left(
\frac{x}{q^{x/2}}
  \sum_{\substack{r \in \upone A \\ U \leq  \deg r \leq x/2}}
    \frac{1}{|r|} \# \{ p \in \mc P: r^2 | a^2 - 4up\}
\right).
$$

Applying Theorem \ref{thm:Effective PNTAP}, we rewrite this as
\begin{equation}\label{eqn:StepFirstTrucation}
O \left(
  q^{x/2}
  \sum_{\substack{r \in \upone A \\ U \leq  \deg r \leq x/2}}
    \frac{1}{|r|\varphi(r^2)}
  +
  \sum_{\substack{r \in \upone A \\ U \leq  \deg r \leq x/2}}
    \frac{(\deg r)^2}{|r|}
\right).
\end{equation}

The second sum in (\ref{eqn:StepFirstTrucation}) is $O(x^3)$. For the first sum in (\ref{eqn:StepFirstTrucation}) we have
$$
\sum_{\substack{r \in \upone A \\ U \leq  \deg r \leq x/2}}
  \frac{1}{|r|\varphi(r^2)}
\leq
\frac{1}{q^{2U}}
\sum_{\substack{r \in \upone A \\ U \leq  \deg r \leq x/2}}
  \frac{1}{\varphi(r)}.
$$

Elementary manipulations give the bound 
$$
\sum_{\substack{r \in \upone A \\ U \leq  \deg r \leq x/2}}
  \frac{1}{\varphi(r)}
\ll
x
\sum_{\substack{r \in \upone A \text{  sqfree}\\ U \leq  \deg r \leq x/2}}
  \frac{1}{\varphi(r)},
$$
as well as
$$
\sum_{\substack{r \in \upone A \text{  sqfree}\\ U \leq  \deg r \leq x/2}}
  \frac{1}{\varphi(r)}
\ll x.
$$

Altogether we have
$$
\frac{C_\infty}{q^{x/2}}
\sum_{p \in \mc P}
  \sum_{\substack{r \in \upone A \\ r^2 | a^2 - 4up \\ \deg r \geq U}}
    \frac{1}{|r|}
      \sum_{\substack{v \in \upone A}}
        \frac{\chi_R(v)}{|v|}
=
O \left(
  \frac{q^{x/2}x^2}{q^{2U}}
  +
  x^3
\right).
$$
Therefore, we require that 
\begin{equation} \label{eqn:LowerBoundonU}
U > \e x \texteqn{for some } 0 < \e < 1.
\end{equation}

We fix a second parameter $V$, depending on $x$ and to be determined precisely later. Our next task is to control the contribution of $v \in \upone A$ with $\deg v \geq V$. We write this as
\begin{equation}\label{StepSecondTruncation}
\frac{C_\infty}{q^{x/2}}
\sum_{p \in \mc P}
  \sum_{\substack{r \in \upone A \\ r^2 | a^2 - 4up \\ \deg r < U}}
    \frac{1}{|r|}
      \sum_{\substack{v \in \upone A \\ \deg v \geq V}}
        \frac{\chi_R(v)}{|v|}
=
O \left(
  \frac{1}{q^{x/2}}
  \sum_{\substack{r \in \upone A \\ r^2 | a^2 - 4up \\ \deg r < U}}
    \frac{1}{|r|}
    \sum_{p \in \mc P}
      \sum_{y=V}^x
        \frac{1}{q^y}
        \sum_{\substack{v \in \upone A \\ \deg v = y}}
          \chi_R(v)
\right).
\end{equation}

The innermost sum on the right hand side of (\ref{StepSecondTruncation}) can be bounded as in the proof of Lemma \ref{lem:CharSumOnRH} by $O \left(q^{y/2} \binom{x - 2 \deg r}{y} \right)$. This leaves us with
$$
\sum_{y=V}^x
  \frac{1}{q^y}
  \sum_{\substack{v \in \upone A \\ \deg v = y}}
    \chi_R(v)
=
O \left(
\sum_{y=V}^x
\frac{1}{q^{y/2}} \binom{x - 2 \deg r}{y}
\right)
=
O \left( \frac{2^x}{q^{V/2}} \right).
$$

This shows that
$$
\sum_{p \in \mc P}
  \sum_{\substack{r \in \upone A \\ r^2 | a^2 - 4up \\ \deg r < U}}
    \frac{1}{|r|}
      \sum_{y=V}^x
        \frac{1}{q^y}
        \sum_{\substack{v \in \upone A \\ \deg v = y}}
          \chi_R(v)
=
O \left(
\frac{2^x}{q^{V/2}} 
\sum_{\substack{r \in \upone A \\ \deg r < U}}
  \frac{1}{|r|}
  \# \{p \in \mc P: r^2|a^2 - 4up  \}
\right).
$$

We apply Theorem \ref{thm:Effective PNTAP} to the inner quantity, and note that
$$
\sum_{\substack{r \in \upone A  \\ \deg r < U}}
\frac{1}{|r| \varphi(r)}
=
O (1)
\texteqn{and}
\sum_{\substack{r \in \upone A  \\ \deg r < U}}
\frac{(\deg r)^2}{|r|}
=
O (U^3),
$$
and (\ref{StepSecondTruncation}) becomes
$$
\frac{C_\infty}{q^{x/2}}
\sum_{p \in \mc P}
  \sum_{\substack{r \in \upone A \\ r^2 | a^2 - 4up \\ \deg r < U}}
    \frac{1}{|r|}
      \sum_{\substack{v \in \upone A \\ \deg v \geq V}}
        \frac{\chi_R(v)}{|v|}
=
O \left(
\frac{2^x}{q^{V/2}} \frac{q^{x/2} + U^3}{x}
\right).
$$

Therefore, we require that 
\begin{equation}\label{eqn:LowerBoundonV}
V > \left(\frac{\log 4}{\log q} + \e \right)x \texteqn{for some } \e > 0.
\end{equation}

\subsection{Extracting the main term}
We are left considering the sum
\begin{equation} \label{eqn:SumBeforePNTAP}
\frac{1}{q^{x/2}}
\sum_{u \in \F^*} 
\sum_{\substack{r \in \upone A \\ \deg r \leq U}} \frac{1}{|r|}
\sum_{\substack{p \in \mc P \\ r^2 \mid a^2 - 4up}}
\sum_{\substack{v \in \upone A \\ \deg v \leq V}}\frac{1}{|v|}\chi_R(v).
\end{equation}

We make some observations leading to conditions in our sum.

\begin{itemize}
\item If $r \mid a^2 - 4up$ and $r \mid a$, then $r \mid p$. But $\deg r \leq U$, and $U$ will be chosen strictly less than $x$, therefore the latter will not happen. This means that we can restrict the sum to a sum over $r$ satisfying
\begin{equation}\label{eqn:CongCondition}
(r,a)=1
\end{equation}.
\item The value of $\chi_R(v)$ depends $R$ modulo $v$ and not on $R$ itself, and $\gcd(R,v) \neq 1$ implies that $\chi_R(v)=0$.
\item Let $\gs$ be a representative for a class in $(A/vA)^*$. Then
$$
\frac{a^2-4up}{r^2} \equiv \gs (\mod v) \iff p \equiv \frac{a^2 - r^2 \gs}{4u} (\mod vr^2).
$$
\item Observe that $\gcd(\gs r^2 - a^2, v) = 1 $ if and only if $\gcd(\gs r^2 - a^2, vr^2) = 1$. One direction is obvious. For the other direction, suppose that $\gcd(\gs r^2 - a^2, v) = 1$ and there exists a prime $\ell$ such that $\ell \mid vr^2$ and $\ell \mid \gs r^2 - a^2$. Our hypothesis implies that $\ell \nmid v$, so it is necessary that $\ell \mid r$ and consequentially that $\ell \mid a^2$. This contradicts (\ref{eqn:CongCondition}).
\end{itemize}

We introduce the explicit conditions just discussed into (\ref{eqn:SumBeforePNTAP}) and partition the inner sum into residue classes to get 

We partition  into residue classes and explicitly introduce the conditions just discussed to write 
\begin{align*}
\frac{1}{q^{x/2}}
\sum_{\substack{r \in \upone A \\ \deg r \leq U \\ \gcd(a,r)=1}}
\sum_{\substack{v \in \upone A \\ \deg v \leq V}}\frac{1}{|rv|}
\sum_{\substack{\gs(v)^* \\ \gcd(\gs r^2 - a^2, v) = 1}}\chi_{\gs}(v)
\# \left\{
p, \deg p = x: p \equiv \frac{a^2 - r^2 \gs}{4u} (\mod vr^2)
  \right\},
\end{align*}
where $*$ indicates that the sum is taken over representatives of $(A / vA )^*$.

We now apply Theorem \ref{thm:Effective PNTAP}. We ensure that $\deg vr^2 < x$ by requiring that 
\begin{equation}\label{eqn:UpperBoundOnUV}
2U + V \leq (1-\e)x \texteqn{ for some} \e > 0.
\end{equation}
 
It follows that
$$
\# \left\{
p, \deg p = x: p \equiv \frac{a^2 - r^2 \gs}{4u} (\mod vr^2)
  \right\}
=
\frac{1}{\varphi(vr^2)} \frac{q^x}{x}
+ O \left(
x^2 q^{x/2}
\right).
$$

We bound the contribution of the error terms as follows.
\begin{align*}
\frac{1}{q^{x/2}} &
\sum_{\substack{r \in \upone A \\ \deg r \leq U \\ \gcd(a,r)=1}}
\sum_{\substack{v \in \upone A \\ \deg v \leq V}}\frac{1}{|rv|}
\sum_{\substack{\gs(v)^* \\ \gcd(\gs r^2 - a^2, v) = 1}}\chi_{\gs}(v)
x^2 q^{x/2}
 \\ &=
O \left( x^2
\sum_{\substack{r \in \upone A \\ \deg r \leq U }}\frac{1}{|r|}
\sum_{\substack{v \in \upone A \\ \deg v \leq V}}\frac{1}{|v|}
\sum_{\gs(v)^*} 1
\right) 
=
O \left( x^4 q^V
\right).
\end{align*}

We see that we need 
\begin{equation}\label{eqn:UpperBoundOnV}
V \leq x(\frac{1}{2} - \e) \texteqn{for some} \e > 0.
\end{equation}

We choose $U$ and $V$ as follows. Under the assumption that $q > 16$, we note that
$$
\frac{\log 4}{\log q} < \frac{1}{2};
$$

provided $x$ is not too small, we can choose 
$$
U := 
\left \lfloor  
\frac{1}{4}x
\right \rfloor,
V := 
\left \lceil
\left( \frac{1}{4} + \frac{\log 2}{\log q} \right) x
\right \rceil
$$
and this choice will satisfy (\ref{eqn:LowerBoundonU}), (\ref{eqn:LowerBoundonV}), (\ref{eqn:UpperBoundOnUV}), and (\ref{eqn:UpperBoundOnV}) as $x \to \infty$. Our main term, coming from the main term in the prime number theorem, becomes

\begin{equation}\label{eqn:MainTermBeforeConstantExtension}
\begin{aligned}
\frac{C_\infty}{q^{x/2}} &
\sum_{\substack{r \in \upone A \\ \deg r \leq U \\ \gcd(a,r)=1}}
\sum_{\substack{v \in \upone A \\ \deg v \leq V}}\frac{1}{|rv|}
\sum_{\substack{\gs(v)^* \\ \gcd(\gs r^2 - a^2, v) = 1}}\chi_{\gs}(v)
\frac{1}{\varphi(vr^2)} \frac{q^x}{x}
 \\ &=
 C_\infty
 \frac{q^{x/2}}{x}
\sum_{\substack{r \in \upone A \\ \deg r \leq U }}
\sum_{\substack{v \in \upone A \\ \deg v \leq V}}\frac{1}{|rv|\varphi(vr^2)}
\sum_{\substack{\gs(v)^* \\ \gcd(\gs r^2 - a^2, v) = 1}}\chi_{\gs}(v)
\end{aligned}
\end{equation}

\section{The constant $C(a)$}

We introduce the following notation. For $v,r \in \upone A$ with $\gcd(r,a) = 1$, write
$$
c(a;v,r) := \sum_{\substack{\gs(v)^* \\ \gcd(\gs r^2 - a^2, v) = 1}}\chi_{\gs}(v)
$$

With this notation, the constant appearing in our the main term of  is
$$
C_\infty
\sum_{\substack{r \in \upone A \\ \deg r \leq U }}
\sum_{\substack{v \in \upone A \\ \deg v \leq V}}\frac{c(a;v,r)}{|rv|\varphi(vr^2)}
$$

We first extend this to a sum over all $v,r \in \upone A$ and show that the infinite sum converges, and converges quickly enough so that we can replace the bounded sum in our main term by the infinite sum. Then, we rewrite the infinite sum as a product over primes in $A$, so that our constant can be more easily compared with those coming from conjectures and other results in the literature.

\subsection{Properties of $c(a; v, r)$ }
Following \cite{DavidPappalardi1999}, we write explicit expressions for $c(a;v,r)$, giving an upper bound on its absolute value. We first note that $c(a;v,r)$ is multiplicative in $v$. Indeed, let $\gcd(v_1, v_2)=1$ and for a fixed $\gs \pmod{ v_1v_2}$ there are unique $\gs_1 \pmod{ v_1}$ and $\gs_2 \pmod{ v_2}$ such that
$$
\gs  = \gs_1 + k_1 v_1 = \gs_2 + k_2 v_2 \texteqn{for some} k_1, k_2 \in A.
$$
Moreover,
$$
\gcd(\gs r^2 - a^2, v_1 v_2) = 1 \implies 
\gcd(\gs r^2 - a^2, v_1) = 1 \texteqn{and} \gcd(\gs r^2 - a^2, v_2) = 1 ,
$$
and for $i=1,2$,
$$
\gcd(\gs r^2 - a^2, v_i) = \gcd((\gs_i + k_i v_i) r^2 - a^2, v_i) = \gcd(\gs_i r^2 - a^2, v_i).
$$

Multiplicativity now follows by splitting the character sum in $c(a;v_1v_2,r)$.

We now focus on $c(a;\ell^k,r)$ where $\ell$ is a prime in $A$. We consider two cases.

If $\gcd(\ell,r)=1$ then $\gs \mapsto \gs r^2$ is a permutation of the set
$$
\{
\gs (\ell^k)^*: \gcd(\gs r^2-a^2, \ell^k)=1.
\}
$$

Therefore
\begin{align*}
c(a;\ell^k,r) = 
\sum_{\substack{\gs(\ell^k)^* \\ \gcd(\gs r^2 - a^2, \ell) = 1}}\chi_{\gs}(\ell)^k
=
\sum_{\substack{\hat\gs(\ell^k)^* \\ \gcd(\hat\gs- a^2, \ell) = 1}}\chi_{\hat\gs}(\ell)^k
=c(a;\ell^k,1) 
\end{align*}

But this last quantity is
\begin{align*}
c(a;\ell^k,1) 
=
|l|^{k-1} \left(
\left(\sum_{\hat\gs(\ell)^*} \chi_{\hat\gs}(\ell)^k \right)
- \chi_{a^2}(\ell)^k
\right).
\end{align*}

If $\ell \nmid a$, this is
$$
c(a;\ell^k,1) 
=
\left\{\begin{array}{ll}
|l|^{k-1} (|\ell|-2) & \texteqn{if} k \texteqn{even,}\\ 
-|l|^{k-1}  & \texteqn{if} k \texteqn{odd,}
\end{array}\right.
$$

If $\ell \mid a$, this is
$$
c(a;\ell^k,1) 
=
\left\{\begin{array}{ll}
|l|^{k-1} (|\ell|-1) & \texteqn{if} k \texteqn{even,}\\ 
0 & \texteqn{if} k \texteqn{odd,}
\end{array}\right.
$$

If $\gcd(\ell,r)=\ell$, it follows that $\ell \nmid a$ and the condition $\gcd(\hat\gs r^2- a^2, \ell) = 1$ is vacuous. Therefore
$$
c(a;\ell^k,r) 
=
\sum_{\hat\gs(\ell)^*} \chi_{\hat\gs}(\ell)^k 
$$

By orthogonality, this is $0$ if $k$ is odd and $|l|^{k-1} (|\ell|-1)$ if $k$ is even. The following lemma summarizes this discussion.

\begin{lemma}\label{lem:ComputedValueOfConstants}
Take the notation above. Then

$$
c(a;\ell^k,r) 
=
\left\{\begin{array}{ll}
c(a;\ell^k,\ell) & \texteqn{if} \ell \mid r\\ 
c(a;\ell^k,1)  & \texteqn{if}  \ell \nmid r;
\end{array}\right.
$$
more precisely, 
$$
c(a;\ell^k,r) 
=
\left\{\begin{array}{ll}
|l|^{k-1} (|\ell|-2) & \texteqn{if} k \texteqn{even,}\ell \nmid a, \ell \nmid r\\ 
-|l|^{k-1}  & \texteqn{if} k \texteqn{odd,}\ell \nmid a, \ell \nmid r \\
|l|^{k-1} (|\ell|-1) & \texteqn{if} k \texteqn{even,} \ell \mid a, \ell \nmid r\\ 
0 & \texteqn{if} k \texteqn{odd,} \ell \mid a, \ell \nmid r\\
|l|^{k-1} (|\ell|-1) & \texteqn{if} k \texteqn{even,}  \ell \mid r \\
0 & \texteqn{if} k \texteqn{odd,} \ell \mid r.\\ 
\end{array}\right.
$$

\end{lemma}

Define $\kappa(\cdot)$ as the multiplicative function with values on prime powers given by
$$
\kappa(\ell^k)
=
\left\{\begin{array}{ll}
1 & \texteqn{if} k \texteqn{even,}\\ 
|\ell| & \texteqn{if} k \texteqn{odd.}
\end{array}\right.
$$

From Lemma \ref{lem:ComputedValueOfConstants} it is clear that
$$
|c(a;\ell^k,r)| \leq \frac{|\ell^k|}{\kappa(\ell^k)}
$$

From the multiplicativity of $c(a;v,r)$, we get the following bound.

\begin{corollary}\label{cor:BoundedValueOfConstants}
Take the notation above. For all $v \in \upone A$, 
$$
|c(a;v,r)| \leq \frac{|v|}{\kappa(v)}
$$
\end{corollary}

\subsection{Extending the constant to an infinite sum}
For fixed $r$, Corollary \ref{cor:BoundedValueOfConstants} implies that
\begin{equation}\label{eqn:BeginExtendingConstant}
\sum_{\substack{v \in \upone A \\ \deg v  \geq V}}\frac{|c(a;v,r)|}{|v|\varphi(v)}
\leq
\sum_{\substack{v \in \upone A \\ \deg v \geq V}}\frac{1}{\kappa(v)\varphi(v)},
\end{equation}

and we wish to control the growth of this sum as $y \to \infty$. If $v \in \upone A$ can be written as $v = m^2n$ with $m,n \in \upone A$ and $n$ squarefree, then $\kappa(v)$ is simply $|n|$. With this we rewrite the sum on the right hand side of (\ref{eqn:BeginExtendingConstant}) as
$$
\sum_{y \geq V}
\sum_{\substack{m \in \upone A \\ \deg m \leq y/2}}
\sum_{\substack{n \in \upone A \text{  sqfree}\\ \deg n = y - 2\deg m}}
\frac{1}{|n|\varphi(m^2n)}
=
\sum_{y \geq V}
\sum_{\substack{m \in \upone A \\ \deg m \leq y/2}}
\sum_{\substack{n \in \upone A \text{  sqfree}\\ \deg n = y - 2\deg m}}
\frac{1}{|mn|\varphi(mn)},
$$
and bound it trivially by
\begin{align*}
O
\left(
\sum_{y \geq V}
  \sum_{\substack{m \in \upone A \\ \deg m \leq y/2}}
    \frac{1}{|m|\varphi(m)}
  \sum_{\substack{n \in \upone A \text{  sqfree}\\ \deg n = y - 2\deg m}}
    \frac{1}{|n|\varphi(n)}
\right)
&=
O
\left(
\sum_{y \geq V}
  \frac{1}{q^y}
  \sum_{\substack{m \in \upone A \\ \deg m \leq y/2}}
    \frac{|m|}{\varphi(m)}
  \sum_{\substack{n \in \upone A \text{  sqfree}\\ \deg n = y - 2\deg m}}
    \frac{1}{\varphi(n)}
\right) \\
&=
O
\left(
\sum_{y \geq V}
  \frac{y}{q^y}
  \sum_{\substack{m \in \upone A \\ \deg m \leq y/2}}
    \frac{|m|}{\varphi(m)}
  \right)
\end{align*}

We bound $\varphi$ crudely from below as $\varphi(m) \gg (q-1)^{\deg m}$, using the fact that $q>2$ since $q$ is an odd prime power. With this bound we get that
$$
 \frac{|m|}{\varphi(m)} \ll \left( \frac{3}{2} \right)^{\deg m}.
$$

Altogether we have
$$
\sum_{\substack{v \in \upone A \\ \deg v  \geq V}}\frac{|c(a;v,r)|}{|v|\varphi(v)}
=
O \left(
\sum_{y \geq V}
y
\frac{\left( \frac{3}{2} \right)^{y}}{q^{y/2}}
\right).
$$

Since $q$ is an odd prime power, this sum converges, and we have the bound
\begin{equation}\label{eqn:ConstantBoundPreliminary}
\sum_{\substack{v \in \upone A \\ \deg v  \geq V}}\frac{|c(a;v,r)|}{|v|\varphi(v)}
=
O \left(
\frac{1}{\left( \frac{2}{3}q^{1/2} \right)^V}
\right)
\end{equation}

With this, we can bound the sums
\begin{equation}\label{eqn:InfiniteSumsConstant}
\sum_{\substack{r \in \upone A \\ \gcd(r,a)=1}}
\sum_{\substack{v \in \upone A \\ \deg v > V}}
\frac{c(a;v,r)}{|rv|\varphi(vr^2)}
\hspace{8pt}\texteqn{and}
\sum_{\substack{r \in \upone A \\ \gcd(r,a)=1 \\ \deg r > U}}
\sum_{\substack{v \in \upone A \\ \deg v \leq V}}
\frac{c(a;v,r)}{|rv|\varphi(vr^2)}.
\end{equation}
For the first sum in (\ref{eqn:InfiniteSumsConstant}), we use (\ref{eqn:ConstantBoundPreliminary}) as follows:
\begin{align*}
\sum_{\substack{r \in \upone A \\ \gcd(r,a)=1}}
\sum_{\substack{v \in \upone A \\ \deg v > V}}
\frac{|c(a;v,r)|}{|rv|\varphi(vr^2)}
&\leq
\sum_{\substack{r \in \upone A \\ \gcd(r,a)=1}}
\frac{1}{|r|\varphi(r^2)}
\sum_{\substack{v \in \upone A \\ \deg v > V}}
\frac{|c(a;v,r)|}{|v|\varphi(v)}
\\ &\ll
\frac{1}{\left( \frac{2}{3}q^{1/2} \right)^V}
\sum_{\substack{r \in \upone A \\ \gcd(r,a)=1}}
\frac{1}{|r|\varphi(r^2)}
\end{align*}

The inner sum can be bounded above by
$$
\sum_{\substack{r \in \upone A }} \frac{1}{|r|^2} = O(1).
$$

For the second sum in (\ref{eqn:InfiniteSumsConstant}), we extend the sum over $v$ and note that
$$
\sum_{\substack{r \in \upone A \\ \gcd(r,a)=1 \\ \deg r > U}}
  \frac{1}{|r|\varphi(r^2)} 
  \sum_{\substack{v \in \upone A}}
    \frac{|c(a;v,r)|}{|v|\varphi(v)} 
\ll
\sum_{\substack{r \in \upone A \\ \gcd(r,a)=1 \\ \deg r > U}}
  \frac{1}{|r|^2} 
  \ll
  \frac{1}{q^U}.
$$

We record the result of the computations above in the following lemma.

\begin{lemma}\label{lem:ExtendConstant}
Take the notation above. Then,
$$
\left|
\sum_{\substack{r \in \upone A \\ \gcd(r,a)=1}}
\sum_{\substack{v \in \upone A}}
\frac{c(a;v,r)}{|rv|\varphi(vr^2)}
-
 \sum_{\substack{r \in \upone A \\ \deg r \leq U \\ \gcd(r,a)=1}}
\sum_{\substack{v \in \upone A \\ \deg v \leq V}}
\frac{c(a;v,r)}{|rv|\varphi(vr^2)}
\right| 
\ll
\frac{1}{q^U}
+
\frac{1}{\left( \frac{2}{3}q^{1/2} \right)^V}
$$
\end{lemma}

In particular, the sums over $v$ and $r$ in (\ref{eqn:MainTermBeforeConstantExtension}) can be extended to sums over suitable elements of $\upone A$ of all degree, at the cost of an error of
$$
O \left(
  \frac{q^{x/2} }{x}
  \left(
    \frac{1}{q^U}
    +
    \frac{1}{\left( \frac{2}{3}q^{1/2} \right)^V}
  \right)
\right),
$$

and this expression is $\ul o \left(  \frac{q^{x/2} }{x} \right)$ as $x \to \infty$.

\subsection{Product expansion of the constant}
To complete the proof, we rewrite the sum
$$
\sum_{\substack{r \in \upone A \\ \gcd(r,a)=1}}
  \sum_{\substack{v \in \upone A}}
    \frac{c(a;v,r)}{|rv|\varphi(vr^2)},
$$
in order to see that it gives rise to $C(a)$ as defined in the statement of Theorem \ref{thm:MainTheoremSmallBox}. The computations needed are elementary and reliant on properties of multiplicative functions. In particular, we follow \cite[Proof of Lemma 4.1, p. 16-17]{DavidPappalardi1999}

From (\ref{eqn:ProductPropertyPhi}) we get
$$
\sum_{\substack{r \in \upone A \\ \gcd(r,a)=1}}
\sum_{\substack{v \in \upone A}}
\frac{c(a;v,r)}{|rv|\varphi(vr^2)}
=
\sum_{\substack{r \in \upone A \\ \gcd(r,a)=1}}
\frac{1}{|r|\varphi(r^2)} 
\sum_{\substack{v \in \upone A}}
\frac{c(a;v,r)}{|v|\varphi(v)}
\frac{\varphi(\gcd(v,r^2))}{|\gcd(v,r^2)|}
$$

We expand the inner sum as a product over primes to write
$$
\sum_{\substack{r \in \upone A \\ \gcd(r,a)=1}}
\frac{1}{|r|\varphi(r^2)} 
\prod_\ell \left(
\sum_{\ga \geq 0} 
\frac{c(a;\ell^\ga,r)}{|\ell^\ga|\varphi(\ell^\ga)}
\frac{\varphi(\gcd(\ell^\ga,r^2))}{|\gcd(\ell^\ga,r^2)|}
\right)
$$

Whenever $\ell \nmid r$, this last factor vanishes and $c(a;\ell^\ga,r) = c(a;\ell^\ga,1)$, so we extract a product from the sum to get
$$
\prod_\ell \left(
\sum_{\ga \geq 0} 
\frac{c(a;\ell^\ga,1)}{|\ell^\ga|\varphi(\ell^\ga)}
\right)
\sum_{\substack{r \in \upone A \\ \gcd(r,a)=1}}
\frac{1}{|r|\varphi(r^2)} 
\prod_{\ell \mid r}
\left(
\frac{ 
\sum_{\ga \geq 0} 
\frac{c(a;\ell^\ga,r)}{|\ell^\ga|\varphi(\ell^\ga)}
\frac{\varphi(\gcd(\ell^\ga,r^2))}{|\gcd(\ell^\ga,r^2)|}
}{
\sum_{\ga \geq 0} 
\frac{c(a;\ell^\ga,1)}{|\ell^\ga|\varphi(\ell^\ga)}
}
\right)
$$

The expression inside of this sum is a multiplicative function over $r$, so we can write the sum as a product over primes $\ell' \nmid a$. Since the product inside of the sum runs over primes dividing $r$, $r = \ell^\gb, \gb \neq 0 \implies \ell'=\ell$. The sum in the denominator gives cancellation in the first product, so we write
\begin{equation}\label{eqn:IntermediateStepProductAsConstant}
\prod_{\ell \mid a} \left(
\sum_{\ga \geq 0} 
\frac{c(a;\ell^\ga,1)}{|\ell^\ga|\varphi(\ell^\ga)}
\right)
\prod_{\ell \nmid a} \left(
\sum_{\substack{\ga \geq 0 \\ \gb \geq 0}} 
\frac{1}{|\ell^\gb| \varphi(\ell^{2\gb})}
\frac{c(a;\ell^\ga,\ell^\gb)}{|\ell^\ga|\varphi(\ell^\ga)}
\frac{\varphi(\gcd(\ell^\ga,\ell^{2\gb}))}{|\gcd(\ell^\ga,\ell^{2\gb})|}
\right)
\end{equation}

This last factor is $1 - \frac{1}{|\ell|}$ if $\ga\gb \neq 0$ and $1$ otherwise. Therefore, the remaining sum becomes
\begin{align*}
\sum_{\substack{\ga \geq 0 \\ \gb \geq 0}} 
&\frac{1}{|\ell^\gb| \varphi(\ell^{2\gb})}
\frac{c(a;\ell^\ga,\ell^\gb)}{|\ell^\ga|\varphi(\ell^\ga)}
\frac{\varphi(\gcd(\ell^\ga,\ell^{2\gb}))}{|\gcd(\ell^\ga,\ell^{2\gb})|}
\\ &=
1 
+ 
\sum_{\ga \geq 1} 
\frac{c(a;\ell^\ga,1)}{|\ell^\ga|\varphi(\ell^\ga)}
+
\sum_{\gb \geq 1} 
\frac{1}{|\ell^\gb| \varphi(\ell^{2\gb})}
+
\left( 1 - \frac{1}{|\ell|} \right)
\sum_{\gb \geq 1} 
\frac{1}{|\ell^\gb| \varphi(\ell^{2\gb})}
\sum_{\ga \geq 1} 
\frac{c(a;\ell^\ga,\ell)}{|\ell^\ga|\varphi(\ell^\ga)}
\\ &=
1 
+ 
\sum_{\ga \geq 1} 
\frac{c(a;\ell^\ga,1)}{|\ell^\ga|\varphi(\ell^\ga)}
+ 
\frac{1}{|\ell|^3-1} 
\left( 
\frac{|\ell|}{|\ell|-1} 
+ 
\sum_{\ga \geq 1} 
\frac{c(a;\ell^\ga,\ell)}{|\ell^\ga|\varphi(\ell^\ga)}
\right)
\end{align*}

Using Lemma \ref{lem:ComputedValueOfConstants}, we compute the remaining sums:
$$
\sum_{\ga \geq 1} 
\frac{c(a;\ell^\ga,1)}{|\ell^\ga|\varphi(\ell^\ga)}
=
\left\{\begin{array}{ll}
\frac{1}{|\ell|^2-1}  & \texteqn{if} \ell \mid a\\ 
\frac{-2}{(|\ell|^2-1)(|\ell|-1)}  & \texteqn{if} \ell \nmid a,
\end{array}\right.
$$
and
$$
\sum_{\ga \geq 1} 
\frac{c(a;\ell^\ga,\ell)}{|\ell^\ga|\varphi(\ell^\ga)}
=
\frac{1}{|\ell|^2-1} 
$$

Returning to \ref{eqn:IntermediateStepProductAsConstant}, we now have
$$
\prod_{\ell \mid a}
\left(
1 + \frac{1}{|\ell|^2-1} 
\right)
\prod_{\ell \nmid a}
\left(
1 - \frac{2}{(|\ell|^2-1)(|\ell|-1)}  
+
\frac{1}{|\ell|^3-1} 
\left( 
\frac{|\ell|}{|\ell|-1} 
+ 
\frac{1}{|\ell|^2-1} 
\right) 
\right),
$$

which simplifies to
$$
\prod_{\ell \mid a}
\left(
1 - \frac{1}{|\ell|^2}
\right) ^{-1}
\prod_{\ell \nmid a}
\left(
\frac{|\ell|(|\ell|^2 - |\ell|-1)}{(|\ell|^2-1)(|\ell|-1)}
\right).
$$

\section*{Acknowledgements}
The author thanks A. C. Cojocaru for suggesting this problem and for conversations about it, Matei Vlad for helpful suggestions regarding character sum estimates, and an anonymous referee for helpful comments on an earlier version of this paper. This project is part of the author's doctoral thesis at the University of Illinois at Chicago, and was completed while the author was on an Abraham Lincoln Fellowship; the author thanks the Department of Mathematics, Statistics, and Computer Science and the Graduate College at the University of Illinois at Chicago for providing conditions suitable for research.

\bibliographystyle{amsalpha}
\bibliography{drinfeldmodules}

\def\cprime{$'$} \def\cprime{$'$} \def\cprime{$'$} \def\cprime{$'$}
  \def\cprime{$'$}
\providecommand{\bysame}{\leavevmode\hbox to3em{\hrulefill}\thinspace}
\providecommand{\MR}{\relax\ifhmode\unskip\space\fi MR }
\providecommand{\MRhref}[2]{%
  \href{http://www.ams.org/mathscinet-getitem?mr=#1}{#2}
}
\providecommand{\href}[2]{#2}
\begin{thebibliography}{MMS88}

\bibitem[Art24]{Artin1924}
E.~Artin, \emph{Quadratische körper im gebiete der höheren kongruenzen. i.},
  Mathematische Zeitschrift \textbf{19} (1924), no.~1, 153--206 (German).

\bibitem[Bai07]{Baier2007}
Stephan Baier, \emph{The {L}ang-{T}rotter conjecture on average}, J. Ramanujan
  Math. Soc. \textbf{22} (2007), no.~4, 299--314. \MR{2376806 (2008j:11065)}

\bibitem[BS09]{BanksShparlinkski2009}
William~D. Banks and Igor~E. Shparlinski, \emph{Sato-{T}ate, cyclicity, and
  divisibility statistics on average for elliptic curves of small height},
  Israel J. Math. \textbf{173} (2009), 253--277. \MR{2570668 (2011a:11121)}

\bibitem[CFM05]{CojocaruFouvryMurty2005}
Alina~Carmen Cojocaru, Etienne Fouvry, and M.~Ram Murty, \emph{The square sieve
  and the {L}ang-{T}rotter conjecture}, Canad. J. Math. \textbf{57} (2005),
  no.~6, 1155--1177. \MR{2178556 (2006e:11074)}

\bibitem[Cha08]{Chang2008}
Mei-Chu Chang, \emph{On a question of {D}avenport and {L}ewis and new character
  sum bounds in finite fields}, Duke Math. J. \textbf{145} (2008), no.~3,
  409--442. \MR{2462111 (2009i:11099)}

\bibitem[Dav96]{David1996}
Chantal David, \emph{Average distribution of supersingular {D}rinfel\cprime d
  modules}, J. Number Theory \textbf{56} (1996), no.~2, 366--380. \MR{1373559
  (96j:11080)}

\bibitem[DL63]{DavenportLewis1963}
H.~Davenport and D.~J. Lewis, \emph{Character sums and primitive roots in
  finite fields}, Rend. Circ. Mat. Palermo (2) \textbf{12} (1963), 129--136.
  \MR{0167482 (29 \#4755)}

\bibitem[DP99]{DavidPappalardi1999}
Chantal David and Francesco Pappalardi, \emph{Average {F}robenius distributions
  of elliptic curves}, Internat. Math. Res. Notices (1999), no.~4, 165--183.
  \MR{1677267 (2000g:11045)}

\bibitem[Dri74]{Drinfeld1974}
V.~G. Drinfel{\cprime}d, \emph{Elliptic modules}, Mat. Sb. (N.S.)
  \textbf{94(136)} (1974), 594--627, 656. \MR{0384707 (52 \#5580)}

\bibitem[Dri77]{Drinfeld1977}
\bysame, \emph{Elliptic modules. {II}}, Mat. Sb. (N.S.) \textbf{102(144)}
  (1977), no.~2, 182--194, 325. \MR{0439758 (55 \#12644)}

\bibitem[FM96]{FouvryMurty1996}
Etienne Fouvry and M.~Ram Murty, \emph{On the distribution of supersingular
  primes}, Canad. J. Math. \textbf{48} (1996), no.~1, 81--104. \MR{1382477
  (97a:11084)}

\bibitem[Gek91]{Gekeler1991}
Ernst-Ulrich Gekeler, \emph{On finite {D}rinfel\cprime d modules}, J. Algebra
  \textbf{141} (1991), no.~1, 187--203. \MR{1118323 (92e:11064)}

\bibitem[Gek08]{Gekeler2008}
\bysame, \emph{Frobenius distributions of {D}rinfeld modules over finite
  fields}, Trans. Amer. Math. Soc. \textbf{360} (2008), no.~4, 1695--1721.
  \MR{2366959 (2008m:11114)}

\bibitem[LT76]{LangTrotter1976}
Serge Lang and Hale Trotter, \emph{Frobenius distributions in {${\rm
  GL}_{2}$}-extensions}, Lecture Notes in Mathematics, Vol. 504,
  Springer-Verlag, Berlin-New York, 1976, Distribution of Frobenius
  automorphisms in ${{\rm{G}}L}_{2}$-extensions of the rational numbers.
  \MR{0568299 (58 \#27900)}

\bibitem[MMS88]{MurtyMurtySaradha1988}
M.~Ram Murty, V.~Kumar Murty, and N.~Saradha, \emph{Modular forms and the
  {C}hebotarev density theorem}, Amer. J. Math. \textbf{110} (1988), no.~2,
  253--281. \MR{935007 (89d:11036)}

\bibitem[Ros02]{Rosen2002}
Michael Rosen, \emph{Number theory in function fields}, Graduate Texts in
  Mathematics, vol. 210, Springer-Verlag, New York, 2002. \MR{1876657
  (2003d:11171)}

\bibitem[Ser81]{Serre1981}
Jean-Pierre Serre, \emph{Quelques applications du th\'eor\`eme de densit\'e de
  {C}hebotarev}, Inst. Hautes \'Etudes Sci. Publ. Math. (1981), no.~54,
  323--401. \MR{644559 (83k:12011)}

\bibitem[Yu95]{Yu1995}
Jiu-Kang Yu, \emph{A class number relation over function fields}, J. Number
  Theory \textbf{54} (1995), no.~2, 318--340. \MR{1354056 (96i:11128)}

\end{thebibliography}
\end{document}